\documentclass[10pt]{article}
\usepackage{amsmath}
\usepackage{amssymb}
\usepackage{amsthm}
\usepackage{mathrsfs}
\numberwithin{equation}{section}

\begin{document}
\title{The boundedness of intrinsic square functions on the weighted Herz spaces}
\author{Hua Wang \footnote{E-mail address: wanghua@pku.edu.cn.}\\
\footnotesize{Department of Mathematics, Zhejiang University, Hangzhou 310027, China}}
\date{}
\maketitle

\begin{abstract}
In this paper, we will obtain the strong type and weak type estimates of intrinsic square functions including the Lusin area integral, Littlewood-Paley $g$-function and $g^*_\lambda$-function on the weighted Herz spaces $\dot K^{\alpha,p}_q(w_1,w_2)$ ($K^{\alpha,p}_q(w_1,w_2)$) with general weights. \\
MSC(2010): 42B25; 42B35 \\
Keywords: Intrinsic square functions; weighted Herz spaces; weighted weak Herz spaces; $A_p$ weights
\end{abstract}

\section{Introduction and main results}

Let ${\mathbb R}^{n+1}_+=\mathbb R^n\times(0,\infty)$ and $\varphi_t(x)=t^{-n}\varphi(x/t)$. The classical square function (Lusin area integral) is a familiar object. If $u(x,t)=P_t*f(x)$ is the Poisson integral of $f$, where $P_t(x)=c_n\frac{t}{(t^2+|x|^2)^{{(n+1)}/2}}$ denotes the Poisson kernel in ${\mathbb R}^{n+1}_+$. Then we define the classical square function (Lusin area integral) $S(f)$ by (see \cite{gundy} and \cite{stein})
\begin{equation*}
S(f)(x)=\bigg(\iint_{\Gamma(x)}\big|\nabla u(y,t)\big|^2t^{1-n}\,dydt\bigg)^{1/2},
\end{equation*}
where $\Gamma(x)$ denotes the usual cone of aperture one:
\begin{equation*}
\Gamma(x)=\big\{(y,t)\in{\mathbb R}^{n+1}_+:|x-y|<t\big\}
\end{equation*}
and
\begin{equation*}
\big|\nabla u(y,t)\big|^2=\left|\frac{\partial u}{\partial t}\right|^2+\sum_{j=1}^n\left|\frac{\partial u}{\partial y_j}\right|^2.
\end{equation*}
Similarly, we can define a cone of aperture $\gamma$ for any $\gamma>0$:
\begin{equation*}
\Gamma_\gamma(x)=\big\{(y,t)\in{\mathbb R}^{n+1}_+:|x-y|<\gamma t\big\},
\end{equation*}
and corresponding square function
\begin{equation*}
S_\gamma(f)(x)=\bigg(\iint_{\Gamma_\gamma(x)}\big|\nabla u(y,t)\big|^2t^{1-n}\,dydt\bigg)^{1/2}.
\end{equation*}
The Littlewood-Paley $g$-function (could be viewed as a ``zero-aperture" version of $S(f)$) and the $g^*_\lambda$-function (could be viewed as an ``infinite aperture" version of $S(f)$) are defined respectively by
(see, for example, \cite{muckenhoupt2} and \cite{stein2})
\begin{equation*}
g(f)(x)=\bigg(\int_0^\infty\big|\nabla u(x,t)\big|^2 t\,dt\bigg)^{1/2}
\end{equation*}
and
\begin{equation*}
g^*_\lambda(f)(x)=\left(\iint_{{\mathbb R}^{n+1}_+}\bigg(\frac t{t+|x-y|}\bigg)^{\lambda n}\big|\nabla u(y,t)\big|^2 t^{1-n}\,dydt\right)^{1/2}, \quad \lambda>1.
\end{equation*}

The modern (real-variable) variant of $S_\gamma(f)$ can be defined in the following way (here we drop the subscript $\gamma$ if $\gamma=1$). Let $\psi\in C^\infty(\mathbb R^n)$ be real, radial, have support contained in $\{x:|x|\le1\}$, and $\int_{\mathbb R^n}\psi(x)\,dx=0$. The continuous square function $S_{\psi,\gamma}(f)$ is defined by (see, for example, \cite{chang} and \cite{chanillo})
\begin{equation*}
S_{\psi,\gamma}(f)(x)=\bigg(\iint_{\Gamma_\gamma(x)}\big|f*\psi_t(y)\big|^2\frac{dydt}{t^{n+1}}\bigg)^{1/2}.
\end{equation*}

In 2007, Wilson \cite{wilson1} introduced a new square function called intrinsic square function which is universal in a sense (see also \cite{wilson2}). This function is independent of any particular kernel $\psi$, and it dominates pointwise all the above-defined square functions. On the other hand, it is not essentially larger than any particular $S_{\psi,\gamma}(f)$. For $0<\beta\le1$, let ${\mathcal C}_\beta$ be the family of functions $\varphi$ defined on $\mathbb R^n$ such that $\varphi$ has support containing in $\{x\in\mathbb R^n: |x|\le1\}$, $\int_{\mathbb R^n}\varphi(x)\,dx=0$, and for all $x, x'\in \mathbb R^n$,
\begin{equation*}
|\varphi(x)-\varphi(x')|\le|x-x'|^\beta.
\end{equation*}
For $(y,t)\in {\mathbb R}^{n+1}_+$ and $f\in L^1_{{loc}}(\mathbb R^n)$, we set
\begin{equation}
A_\beta(f)(y,t)=\sup_{\varphi\in{\mathcal C}_\beta}\big|f*\varphi_t(y)\big|=\sup_{\varphi\in{\mathcal C}_\beta}\bigg|\int_{\mathbb R^n}\varphi_t(y-z)f(z)\,dz\bigg|.
\end{equation}
Then we define the intrinsic square function of $f$ (of order $\beta$) by the formula
\begin{equation}
\mathcal S_\beta(f)(x)=\left(\iint_{\Gamma(x)}\Big(A_\beta(f)(y,t)\Big)^2\frac{dydt}{t^{n+1}}\right)^{1/2}.
\end{equation}
We can also define varying-aperture versions of $\mathcal S_\beta(f)$ by the formula
\begin{equation}
\mathcal S_{\beta,\gamma}(f)(x)=\left(\iint_{\Gamma_\gamma(x)}\Big(A_\beta(f)(y,t)\Big)^2\frac{dydt}{t^{n+1}}\right)^{1/2}.
\end{equation}
The intrinsic Littlewood-Paley $\mathcal G$-function and the intrinsic $\mathcal G^*_\lambda$-function will be given respectively by
\begin{equation}
\mathcal G_\beta(f)(x)=\left(\int_0^\infty\Big(A_\beta(f)(x,t)\Big)^2\frac{dt}{t}\right)^{1/2}
\end{equation}
and
\begin{equation}
\mathcal G^*_{\lambda,\beta}(f)(x)=\left(\iint_{{\mathbb R}^{n+1}_+}\left(\frac t{t+|x-y|}\right)^{\lambda n}\Big(A_\beta(f)(y,t)\Big)^2\frac{dydt}{t^{n+1}}\right)^{1/2}, \lambda>1.
\end{equation}

In \cite{wilson2}, Wilson showed the following weighted $L^p$ boundedness of the intrinsic square functions.

\newtheorem*{thma}{Theorem A}

\begin{thma}
Let $0<\beta\le1$, $1<p<\infty$ and $w\in A_p (\mbox{Muckenhoupt weight class})$. Then there exists a constant $C>0$ independent of $f$ such that
\begin{equation*}
\|\mathcal S_\beta(f)\|_{L^p_w}\le C \|f\|_{L^p_w}.
\end{equation*}
\end{thma}

Moreover, in \cite{lerner}, Lerner obtained sharp $L^p_w$ norm inequalities for the intrinsic square functions in terms of the $A_p$ characteristic constant of $w$ for all $1<p<\infty$. For further discussions about the boundedness of intrinsic square functions on various function spaces, we refer the readers to \cite{huang,wang1,wang2,wang3,wang4}.

Before stating our main results, let us first recall some definitions about the weighted Herz and weak Herz spaces. For more information about these spaces, one can see \cite{komori,lu1,lu4,lu3,tang} and the references therein. Let $B_k=B(0,2^k)=\{x\in\mathbb R^n:|x|\le 2^k\}$ and $C_k=B_k\backslash B_{k-1}$ for any $k\in\mathbb Z$. Denote $\chi_k=\chi_{_{C_k}}$ for $k\in\mathbb Z$, $\widetilde\chi_k=\chi_k$ if $k\in\mathbb N$ and $\widetilde\chi_0=\chi_{_{B_0}}$, where $\chi_{_{E}}$ is the characteristic function of the set $E$. For any given weight function $w$ on $\mathbb R^n$ and $0<q<\infty$, we denote by $L^q_w(\mathbb R^n)$ the space of all functions $f$ satisfying
\begin{equation}
\|f\|_{L^q_w}=\left(\int_{\mathbb R^n}|f(x)|^qw(x)\,dx\right)^{1/q}<\infty.
\end{equation}

\newtheorem{definition}{Definition}[section]

\begin{definition}[\cite{lu1}]
Let $\alpha\in\mathbb R$, $0<p, q<\infty$ and $w_1$, $w_2$ be two weight functions on $\mathbb R^n$.

$(a)$ The homogeneous weighted Herz space $\dot K^{\alpha,p}_q(w_1,w_2)$ is defined by
\begin{equation*}
\dot K^{\alpha,p}_q(w_1,w_2)=\Big\{f\in L^q_{loc}(\mathbb R^n\backslash\{0\},w_2):\big\|f\big\|_{\dot K^{\alpha,p}_q(w_1,w_2)}<\infty\Big\},
\end{equation*}
where
\begin{equation}
\big\|f\big\|_{\dot K^{\alpha,p}_q(w_1,w_2)}=\left(\sum_{k\in\mathbb Z}\big[w_1(B_k)\big]^{{\alpha p}/n}\big\|f\chi_k\big\|_{L^q_{w_2}}^p\right)^{1/p}.
\end{equation}

$(b)$ The non-homogeneous weighted Herz space $K^{\alpha,p}_q(w_1,w_2)$ is defined by
\begin{equation*}
K^{\alpha,p}_q(w_1,w_2)=\Big\{f\in L^q_{loc}(\mathbb R^n,w_2):\big\|f\big\|_{K^{\alpha,p}_q(w_1,w_2)}<\infty\Big\},
\end{equation*}
where
\begin{equation}
\big\|f\big\|_{K^{\alpha,p}_q(w_1,w_2)}=\left(\sum_{k=0}^\infty\big[w_1(B_k)\big]^{{\alpha p}/n}\big\|f\widetilde{\chi}_k\big\|_{L^q_{w_2}}^p\right)^{1/p}.
\end{equation}
\end{definition}

For any $k\in\mathbb Z$, $\lambda>0$ and any measurable function $f$ on $\mathbb R^n$, we set $E_k(\lambda,f)=\{x\in C_k:|f(x)|>\lambda\}$. Let $\widetilde E_k(\lambda,f)=E_k(\lambda,f)$ for $k\in\mathbb N$ and $\widetilde E_0(\lambda,f)=\{x\in B(0,1):|f(x)|>\lambda\}$.

\begin{definition}[\cite{lu3}]
Let $\alpha\in\mathbb R$, $0<p, q<\infty$ and $w_1$, $w_2$ be two weight functions on $\mathbb R^n$.

$(c)$ A measurable function $f(x)$ on $\mathbb R^n$ is said to belong to the homogeneous weighted weak Herz space $W\dot K^{\alpha,p}_q(w_1,w_2)$ if
\begin{equation}
\big\|f\big\|_{W\dot K^{\alpha,p}_q(w_1,w_2)}=\sup_{\lambda>0}\lambda\left(\sum_{k\in\mathbb Z}\big[w_1(B_k)\big]^{{\alpha p}/n}\big[w_2(E_k(\lambda,f))\big]^{p/q}\right)^{1/p}<\infty.
\end{equation}

$(d)$ A measurable function $f(x)$ on $\mathbb R^n$ is said to belong to the non-homogeneous weighted weak Herz space $W K^{\alpha,p}_q(w_1,w_2)$ if
\begin{equation}
\big\|f\big\|_{W K^{\alpha,p}_q(w_1,w_2)}=\sup_{\lambda>0}\lambda\left(\sum_{k=0}^\infty\big[w_1(B_k)\big]^{{\alpha p}/n}\big[w_2(\widetilde E_k(\lambda,f))\big]^{p/q}\right)^{1/p}<\infty.
\end{equation}
\end{definition}

Obviously, if $\alpha=0$, then $\dot K^{0,q}_q(w_1,w_2)=K^{0,q}_q(w_1,w_2)=L^q_{w_2}(\mathbb R^n)$ for any $0<q<\infty$. We also have $W\dot K^{0,q}_q(w_1,w_2)=WK^{0,q}_q(w_1,w_2)=WL^q_{w_2}(\mathbb R^n)$ when $\alpha=0$ and $0<q<\infty$, where
\begin{equation}
\|f\|_{WL^q_w}=\sup_{\lambda>0}\lambda\cdot w\big(\big\{x\in\mathbb R^n:|f(x)|>\lambda\big\}\big)^{1/q}<\infty.
\end{equation}

Thus, weighted (weak) Herz spaces are generalizations of the weighted (weak) Lebesgue spaces.
The main purpose of this paper is to consider the boundedness of intrinsic square functions on weighted Herz spaces with $A_p$ weights. At the extreme case, we will also prove that these operators are bounded from the weighted Herz spaces to the weighted weak Herz spaces. Our main results in the paper are formulated as follows.

\newtheorem{theorem}{Theorem}[section]

\begin{theorem}
Let $0<\beta\le1$, $0<p<\infty$, $1<q<\infty$, $w_1\in A_{q_1}$ and $w_2\in A_{q_2}$. Then $\mathcal S_\beta$ is bounded on $\dot K^{\alpha,p}_q(w_1,w_2)$ $(K^{\alpha,p}_q(w_1,w_2))$ provided that $w_1$ and $w_2$ satisfy either of the following

$(i)$ $w_1=w_2$, $1\le q_1=q_2\le q$ and $-{nq_1}/q<\alpha q_1<n(1-{q_2}/q);$

$(ii)$ $w_1\neq w_2$, $1\le q_1<\infty$, $1\le q_2\le q$ and $0<\alpha q_1<n(1-{q_2}/q)$.
\end{theorem}

\begin{theorem}
Let $0<\beta\le1$, $0<p\le1$, $1<q<\infty$, $w_1\in A_{q_1}$ and $w_2\in A_{q_2}$. If $1\le q_1<\infty$, $1\le q_2\le q$ and $\alpha q_1=n(1-{q_2}/q)$, then $\mathcal S_\beta$ is bounded from $\dot K^{\alpha,p}_q(w_1,w_2)$ $(K^{\alpha,p}_q(w_1,w_2))$ into $W\dot K^{\alpha,p}_q(w_1,w_2)$ $(WK^{\alpha,p}_q(w_1,w_2))$.
\end{theorem}

\begin{theorem}
Let $0<\beta\le1$, $0<p<\infty$, $1<q<\infty$, $w_1\in A_{q_1}$ and $w_2\in A_{q_2}$. If $\lambda>\max\{q_2,3\}$, then $\mathcal G^*_{\lambda,\beta}$ is bounded on $\dot K^{\alpha,p}_q(w_1,w_2)$ $(K^{\alpha,p}_q(w_1,w_2))$ provided that $w_1$ and $w_2$ satisfy either of the following

$(i)$ $w_1=w_2$, $1\le q_1=q_2\le q$ and $-{nq_1}/q<\alpha q_1<n(1-{q_2}/q);$

$(ii)$ $w_1\neq w_2$, $1\le q_1<\infty$, $1\le q_2\le q$ and $0<\alpha q_1<n(1-{q_2}/q)$.
\end{theorem}

\begin{theorem}
Let $0<\beta\le1$, $0<p\le1$, $1<q<\infty$, $w_1\in A_{q_1}$ and $w_2\in A_{q_2}$. If $1\le q_1<\infty$, $1\le q_2\le q$, $\alpha q_1=n(1-{q_2}/q)$ and $\lambda>\max\{q_2,3\}$, then $\mathcal G^*_{\lambda,\beta}$ is bounded from $\dot K^{\alpha,p}_q(w_1,w_2)$ $(K^{\alpha,p}_q(w_1,w_2))$ into $W\dot K^{\alpha,p}_q(w_1,w_2)$ $(WK^{\alpha,p}_q(w_1,w_2))$.
\end{theorem}

In \cite{wilson1}, Wilson also showed that for any $0<\beta\le1$, the functions $\mathcal S_\beta(f)(x)$ and $\mathcal G_\beta(f)(x)$ are pointwise comparable, with comparability constants depending only on $\beta$ and $n$. Thus, as a direct consequence of Theorems 1.1 and 1.2, we obtain the following:

\newtheorem{corollary}[theorem]{Corollary}

\begin{corollary}
Let $0<\beta\le1$, $0<p<\infty$, $1<q<\infty$,$w_1\in A_{q_1}$ and $w_2\in A_{q_2}$. Then $\mathcal G_\beta$ is bounded on $\dot K^{\alpha,p}_q(w_1,w_2)$ $(K^{\alpha,p}_q(w_1,w_2))$ provided that $w_1$ and $w_2$ satisfy either of the following

$(i)$ $w_1=w_2$, $1\le q_1=q_2\le q$ and $-{nq_1}/q<\alpha q_1<n(1-{q_2}/q);$

$(ii)$ $w_1\neq w_2$, $1\le q_1<\infty$, $1\le q_2\le q$ and $0<\alpha q_1<n(1-{q_2}/q)$.
\end{corollary}

\begin{corollary}
Let $0<\beta\le1$, $0<p\le1$, $1<q<\infty$, $w_1\in A_{q_1}$ and $w_2\in A_{q_2}$. If $1\le q_1<\infty$, $1\le q_2\le q$ and $\alpha q_1=n(1-{q_2}/q)$, then $\mathcal G_\beta$ is bounded from $\dot K^{\alpha,p}_q(w_1,w_2)$ $(K^{\alpha,p}_q(w_1,w_2))$ into $W\dot K^{\alpha,p}_q(w_1,w_2)$ $(WK^{\alpha,p}_q(w_1,w_2))$.
\end{corollary}

\section{$A_p$ weights}

The classical $A_p$ weight theory was first introduced by Muckenhoupt in the study of weighted $L^p$ boundedness of Hardy-Littlewood maximal functions in \cite{muckenhoupt1}. A weight $w$ is a nonnegative, locally integrable function on $\mathbb R^n$, $B=B(x_0,r_B)$ denotes the ball with the center $x_0$ and radius $r_B$. For any ball $B$ and $\lambda>0$, $\lambda B$ denotes the ball concentric with $B$ whose radius is $\lambda$ times as long. For a given weight function $w$ and a measurable set $E$, we also denote the Lebesgue measure of $E$ by $|E|$ and set weighted measure $w(E)=\int_E w(x)\,dx$. We say that $w$ is in the Muckenhoupt class $A_p$ with $1<p<\infty$, if there exists a constant $C>0$ such that for every ball $B\subseteq \mathbb R^n$,
\begin{equation}
\left(\frac1{|B|}\int_B w(x)\,dx\right)\left(\frac1{|B|}\int_B w(x)^{-1/{(p-1)}}\,dx\right)^{p-1}\le C.
\end{equation}
For the endpoint case $p=1$, $w\in A_1$, if
\begin{equation}
\frac1{|B|}\int_B w(x)\,dx\le C\cdot\underset{x\in B}{\mbox{ess\,inf}}\,w(x)\quad\mbox{for every ball}\;B\subseteq\mathbb R^n,
\end{equation}
where $C$ is a positive constant which is independent of the choice of $B$. The smallest value of $C$ such that the above inequalities hold is called the $A_p$ characteristic constant of $w$ and denoted by $[w]_{A_p}$. If there exist two constants $r>1$ and $C>0$ such that the following reverse H\"older inequality holds
\begin{equation}
\left(\frac{1}{|B|}\int_B w(x)^r\,dx\right)^{1/r}\le C\left(\frac{1}{|B|}\int_B w(x)\,dx\right) \quad\mbox{for every ball}\;B\subseteq\mathbb R^n,
\end{equation}
then we say that $w$ satisfies the reverse H\"older condition of order $r$ and write $w\in RH_r$.
It is well known that if $w\in A_p$ with $1\le p<\infty$, then $w\in A_q$ for all $q>p$. Moreover, if $w\in A_p$ with $1\le p<\infty$, then there exists $r>1$ such that $w\in RH_r$.

The following properties for $A_p$ weights will be repeatedly used in this paper.

\newtheorem{lemma}[theorem]{Lemma}

\begin{lemma}[\cite{garcia2}]
Let $w\in A_p$ with $p\ge1$. Then, for any ball $B$, there exists an absolute constant $C>0$ such that
\begin{equation}
w(2B)\le C\,w(B).
\end{equation}
In general, for any $\lambda>1$, we have
\begin{equation}
w(\lambda B)\le C\cdot\lambda^{np}w(B),
\end{equation}
where $C$ does not depend on $B$ nor on $\lambda$.
\end{lemma}

\begin{lemma}[\cite{garcia2,gundy}]
Let $w\in A_p\cap RH_r$, $p\ge1$ and $r>1$. Then there exist constants $C_1$, $C_2>0$ such that
\begin{equation}
C_1\left(\frac{|E|}{|B|}\right)^p\le\frac{w(E)}{w(B)}\le C_2\left(\frac{|E|}{|B|}\right)^{(r-1)/r}
\end{equation}
for any measurable subset $E$ of a ball $B$.
\end{lemma}

Throughout this article, $C$ always denotes a positive constant which is independent of
the main parameters involved, but may vary from line to line.

\section{Proofs of Theorems 1.1 and 1.2}

\begin{proof}[Proof of Theorem 1.1]
We only need to show the theorem for the homogeneous case because the proof of the non-homogeneous result is similar and so is omitted here. Let $f\in \dot K^{\alpha,p}_q(w_1,w_2)$. Following \cite{lu2}, for any $k\in\mathbb Z$, we decompose $f(x)$ as
\begin{equation*}
\begin{split}
f(x)&=f(x)\chi_{\{2^{k-2}<|x|\le 2^{k+1}\}}(x)+f(x)\chi_{\{|x|\le 2^{k-2}\}}(x)+f(x)\chi_{\{|x|>2^{k+1}\}}(x)\\
&=f_1(x)+f_2(x)+f_3(x).
\end{split}
\end{equation*}
Since $\mathcal S_\beta$ ($0<\beta\le1$) is a sublinear operator, then we can write
\begin{equation*}
\begin{split}
\big\|\mathcal S_\beta(f)\big\|^p_{\dot K^{\alpha,p}_q(w_1,w_2)}&=\sum_{k\in\mathbb Z}\big[w_1(B_k)\big]^{{\alpha p}/n}\big\|\mathcal S_\beta(f)\chi_k\big\|_{L^q_{w_2}}^p\\
&\le C\sum_{i=1}^3\sum_{k\in\mathbb Z}\big[w_1(B_k)\big]^{{\alpha p}/n}\big\|\mathcal S_\beta(f_i)\chi_k\big\|_{L^q_{w_2}}^p\\
&=I_1+I_2+I_3.
\end{split}
\end{equation*}
Since $w_2\in A_{q_2}$ and $1\le q_2\le q$, then $w_2\in A_q$. By Theorem A and Lemma 2.1, we have
\begin{equation*}
\begin{split}
I_1&\le C\sum_{k\in\mathbb Z}\big[w_1(B_k)\big]^{{\alpha p}/n}\big\|f_1\big\|_{L^q_{w_2}}^p\\
&\le C\sum_{k\in\mathbb Z}\big[w_1(B_k)\big]^{{\alpha p}/n}\big\|f\chi_k\big\|_{L^q_{w_2}}^p\\
&\le C\big\|f\big\|^p_{\dot K^{\alpha,p}_q(w_1,w_2)}.
\end{split}
\end{equation*}
For the term $I_2$, we first use Minkowski's inequality to derive
\begin{equation*}
I_2\le C\sum_{k\in\mathbb Z}\big[w_1(B_k)\big]^{{\alpha p}/n}\left(\sum_{\ell=-\infty}^{k-2}\big\|\mathcal S_\beta(f\chi_{\ell})\chi_k\big\|_{L^q_{w_2}}\right)^p.
\end{equation*}
For any $\varphi\in{\mathcal C}_\beta$, $0<\beta\le1$ and $(y,t)\in\Gamma(x)$, we have
\begin{align}
\big|(f\chi_{\ell})*\varphi_t(y)\big|&=\bigg|\int_{2^{\ell-1}<|z|\le2^{\ell}}\varphi_t(y-z)f(z)\,dz\bigg|\notag\\
&\le C\cdot t^{-n}\int_{\{2^{\ell-1}<|z|\le2^{\ell}\}\cap\{z:|y-z|\le t\}}|f(z)|\,dz.
\end{align}
For any $x\in C_k$, $(y,t)\in\Gamma(x)$ and $z\in\{2^{\ell-1}<|z|\le2^{\ell}\}\cap B(y,t)$ with $\ell\le k-2$, then by a direct computation, we can easily see that
\begin{equation*}
2t\ge |x-y|+|y-z|\ge|x-z|\ge|x|-|z|\ge \frac{|x|}{2}.
\end{equation*}
Thus, by using the above inequality (3.1) and Minkowski's inequality, we deduce
\begin{align}
\big|\mathcal S_\beta(f\chi_{\ell})(x)\big|&=\left(\iint_{\Gamma(x)}\Big(\sup_{\varphi\in{\mathcal C}_\beta}\big|(f\chi_{\ell})*\varphi_t(y)\big|\Big)^2\frac{dydt}{t^{n+1}}\right)^{1/2}\notag\\
&\le C\left(\int_{\frac{|x|}{4}}^\infty\int_{|x-y|<t}\bigg|t^{-n}
\int_{2^{\ell-1}<|z|\le2^{\ell}}|f(z)|\,dz\bigg|^2\frac{dydt}{t^{n+1}}\right)^{1/2}\notag\\
&\le C\bigg(\int_{2^{\ell-1}<|z|\le2^{\ell}}|f(z)|\,dz\bigg)
\left(\int_{\frac{|x|}{4}}^\infty\frac{dt}{t^{2n+1}}\right)^{1/2}\notag\\
&\le C\cdot\frac{1}{|x|^n}\bigg(\int_{2^{\ell-1}<|z|\le2^{\ell}}|f(z)|\,dz\bigg).
\end{align}
Denote the conjugate exponent of $q>1$ by $q'=q/{(q-1)}$. Applying H\"older's inequality and the $A_q$ condition, we can deduce that
\begin{align}
\int_{2^{\ell-1}<|z|\le2^{\ell}}|f(z)|\,dz&\le\bigg(\int_{2^{\ell-1}<|z|\le2^{\ell}}|f(z)|^qw_2(z)\,dz\bigg)^{1/q}
\bigg(\int_{2^{\ell-1}<|z|\le2^{\ell}}w_2(z)^{-{q'}/q}\,dz\bigg)^{1/{q'}}\notag\\
&\le C\cdot|B_{\ell}|\big[w_2(B_{\ell})\big]^{-1/q}\big\|f\chi_{\ell}\big\|_{L^q_{w_2}}.
\end{align}
Substituting the above inequality (3.3) into (3.2), we thus obtain
\begin{equation*}
\begin{split}
I_2&\le C\sum_{k\in\mathbb Z}\big[w_1(B_k)\big]^{{\alpha p}/n}\left(\sum_{\ell=-\infty}^{k-2}\bigg\{\int_{2^{k-1}<|x|\le 2^k}\big|\mathcal S_\beta(f\chi_{\ell})(x)\big|^qw_2(x)\,dx\bigg\}^{1/q}\right)^p\\
&\le C\sum_{k\in\mathbb Z}\big[w_1(B_k)\big]^{{\alpha p}/n}\left(\sum_{\ell=-\infty}^{k-2}|B_{\ell}|\big[w_2(B_{\ell})\big]^{-1/q}\big\|f\chi_{\ell}\big\|_{L^q_{w_2}}
\bigg\{\int_{2^{k-1}<|x|\le 2^k}\frac{w_2(x)}{|x|^{nq}}dx\bigg\}^{1/q}\right)^p\\
&\le C\sum_{k\in\mathbb Z}\big[w_1(B_k)\big]^{{\alpha p}/n}\left(\sum_{\ell=-\infty}^{k-2}\frac{|B_{\ell}|}{|B_k|}\cdot
\frac{[w_2(B_k)]^{1/q}}{[w_2(B_{\ell})]^{1/q}}\big\|f\chi_{\ell}\big\|_{L^q_{w_2}}\right)^p.
\end{split}
\end{equation*}
Here, we shall consider two cases. For the case of $0<p\le1$, using the well-known inequality $\left(\sum_{\ell}|a_\ell|\right)^p\le\sum_{\ell}|a_\ell|^p$ and changing the order of summation, we find that
\begin{equation*}
I_2\le C\sum_{\ell\in\mathbb Z}\big[w_1(B_{\ell})\big]^{{\alpha p}/n}\big\|f\chi_{\ell}\big\|^p_{L^q_{w_2}}
\left(\sum_{k=\ell+2}^\infty\frac{|B_{\ell}|^p}{|B_k|^p}\cdot
\frac{[w_2(B_k)]^{p/q}}{[w_2(B_{\ell})]^{p/q}}\cdot\frac{[w_1(B_k)]^{{\alpha p}/n}}{[w_1(B_{\ell})]^{{\alpha p}/n}}\right).
\end{equation*}
Moreover, it follows immediately from Lemma 2.1 that
\begin{equation*}
I_2\le C\sum_{\ell\in\mathbb Z}\big[w_1(B_{\ell})\big]^{{\alpha p}/n}\big\|f\chi_{\ell}\big\|^p_{L^q_{w_2}}
\left(\sum_{k=\ell+2}^\infty\frac{|B_{\ell}|^p}{|B_k|^p}\cdot
\frac{[w_2(B_k)]^{p/q}}{[w_2(B_{\ell+2})]^{p/q}}
\cdot\frac{[w_1(B_k)]^{{\alpha p}/n}}{[w_1(B_{\ell+2})]^{{\alpha p}/n}}\right).
\end{equation*}
Since $B_k\supseteq B_{\ell+2}$ when $k\ge \ell+2$ and $w_i\in A_{q_i}$ for $i=1, 2$. Then by Lemma 2.2, we can get
\begin{equation}
\frac{w_i(B_k)}{w_i(B_{\ell+2})}\le C\left(\frac{|B_k|}{|B_{\ell+2}|}\right)^{q_i},\quad \mbox{for $i=1$ and 2}.
\end{equation}
Therefore
\begin{equation*}
\begin{split}
I_2&\le C\sum_{\ell\in\mathbb Z}\big[w_1(B_{\ell})\big]^{{\alpha p}/n}\big\|f\chi_{\ell}\big\|^p_{L^q_{w_2}}
\left(\sum_{k=\ell+2}^\infty\left[\frac{|B_{\ell+2}|}{|B_k|}\right]^{p-{\alpha q_1p}/n-{q_2p}/q}\right)\\
&\le C\sum_{\ell\in\mathbb Z}\big[w_1(B_{\ell})\big]^{{\alpha p}/n}\big\|f\chi_{\ell}\big\|^p_{L^q_{w_2}}
\left(\sum_{k=0}^\infty2^{-kn(p-{\alpha q_1p}/n-{q_2p}/q)}\right)\\
&\le C\sum_{\ell\in\mathbb Z}\big[w_1(B_{\ell})\big]^{{\alpha p}/n}\big\|f\chi_{\ell}\big\|^p_{L^q_{w_2}},
\end{split}
\end{equation*}
where the last inequality holds since $\alpha q_1<n(1-{q_2}/q)$. On the other hand, for the case of $1<p<\infty$, we will use H\"older's inequality to obtain
\begin{equation*}
\begin{split}
&\left(\sum_{\ell=-\infty}^{k-2}\frac{|B_{\ell}|}{|B_k|}\cdot
\frac{[w_2(B_k)]^{1/q}}{[w_2(B_{\ell})]^{1/q}}\cdot\big[w_1(B_k)\big]^{\alpha/n}
\big\|f\chi_{\ell}\big\|_{L^q_{w_2}}\right)^p\\
\le &\left(\sum_{\ell=-\infty}^{k-2}\big[w_1(B_{\ell})\big]^{{\alpha p}/n}\big\|f\chi_{\ell}\big\|^p_{L^q_{w_2}}
\frac{|B_{\ell}|^{p/2}}{|B_k|^{p/2}}\cdot\frac{[w_2(B_k)]^{p/{2q}}}{[w_2(B_{\ell})]^{p/{2q}}}
\cdot\frac{[w_1(B_k)]^{{\alpha p}/{2n}}}{[w_1(B_{\ell})]^{{\alpha p}/{2n}}}\right)\\
&\times
\left(\sum_{\ell=-\infty}^{k-2}\frac{|B_{\ell}|^{{p'}/2}}{|B_k|^{{p'}/2}}
\cdot\frac{[w_2(B_k)]^{{p'}/{2q}}}{[w_2(B_{\ell})]^{{p'}/{2q}}}\cdot\frac{[w_1(B_k)]^{{\alpha p'}/{2n}}}{[w_1(B_{\ell})]^{{\alpha p'}/{2n}}}\right)^{p/{p'}}.
\end{split}
\end{equation*}
Using the same arguments as above, we can also prove the following estimates under the assumption that $\alpha q_1<n(1-{q_2}/q)$.
\begin{equation}
\sum_{k=\ell+2}^\infty\frac{|B_{\ell}|^{p/2}}{|B_k|^{p/2}}\cdot\frac{[w_2(B_k)]^{p/{2q}}}{[w_2(B_{\ell})]^{p/{2q}}}
\cdot\frac{[w_1(B_k)]^{{\alpha p}/{2n}}}{[w_1(B_{\ell})]^{{\alpha p}/{2n}}}\le C
\end{equation}
and
\begin{equation}
\sum_{\ell=-\infty}^{k-2}\frac{|B_{\ell}|^{{p'}/2}}{|B_k|^{{p'}/2}}
\cdot\frac{[w_2(B_k)]^{{p'}/{2q}}}{[w_2(B_{\ell})]^{{p'}/{2q}}}\cdot\frac{[w_1(B_k)]^{{\alpha p'}/{2n}}}{[w_1(B_{\ell})]^{{\alpha p'}/{2n}}}\le C.
\end{equation}
Hence
\begin{equation*}
\begin{split}
I_2&\le C\sum_{k\in\mathbb Z}\left(\sum_{\ell=-\infty}^{k-2}\big[w_1(B_{\ell})\big]^{{\alpha p}/n}\big\|f\chi_{\ell}\big\|^p_{L^q_{w_2}}\frac{|B_{\ell}|^{p/2}}{|B_k|^{p/2}}\cdot
\frac{[w_2(B_k)]^{p/{2q}}}{[w_2(B_{\ell})]^{p/{2q}}}
\cdot\frac{[w_1(B_k)]^{{\alpha p}/{2n}}}{[w_1(B_{\ell})]^{{\alpha p}/{2n}}}\right)\\
&\le C\sum_{\ell\in\mathbb Z}\big[w_1(B_{\ell})\big]^{{\alpha p}/n}\big\|f\chi_{\ell}\big\|^p_{L^q_{w_2}}
\left(\sum_{k=\ell+2}^\infty\frac{|B_{\ell}|^{p/2}}{|B_k|^{p/2}}
\cdot\frac{[w_2(B_k)]^{p/{2q}}}{[w_2(B_{\ell})]^{p/{2q}}}
\cdot\frac{[w_1(B_k)]^{{\alpha p}/{2n}}}{[w_1(B_{\ell})]^{{\alpha p}/{2n}}}\right)\\
&\le C\sum_{\ell\in\mathbb Z}\big[w_1(B_{\ell})\big]^{{\alpha p}/n}\big\|f\chi_{\ell}\big\|^p_{L^q_{w_2}}.
\end{split}
\end{equation*}
Summarizing the above estimates for the term $I_2$, we obtain that for every $0<p<\infty$,
\begin{equation*}
I_2\le C\sum_{\ell\in\mathbb Z}\big[w_1(B_{\ell})\big]^{{\alpha p}/n}\big\|f\chi_{\ell}\big\|^p_{L^q_{w_2}}\le C\big\|f\big\|^p_{\dot K^{\alpha,p}_q(w_1,w_2)}.
\end{equation*}
Let us now turn to estimate the last term $I_3$. In this case, for any $x\in C_k$, $(y,t)\in\Gamma(x)$ and $z\in\{2^{\ell-1}<|z|\le2^{\ell}\}\cap B(y,t)$ with $\ell\ge k+2$, it is easy to check that
\begin{equation*}
2t\ge |x-y|+|y-z|\ge|x-z|\ge|z|-|x|\ge \frac{|z|}{2}.
\end{equation*}
Then it follows from the inequality (3.1) and Minkowski's inequality that
\begin{align}
\big|\mathcal S_\beta(f\chi_{\ell})(x)\big|&\le C\left(\int_{\frac{|z|}{4}}^\infty\int_{|x-y|<t}\bigg|t^{-n}
\int_{2^{\ell-1}<|z|\le2^{\ell}}|f(z)|\,dz\bigg|^2\frac{dydt}{t^{n+1}}\right)^{1/2}\notag\\
&\le C\bigg(\int_{2^{\ell-1}<|z|\le2^{\ell}}|f(z)|\,dz\bigg)
\left(\int_{\frac{|z|}{4}}^\infty\frac{dt}{t^{2n+1}}\right)^{1/2}\notag\\
&\le C\bigg(\int_{2^{\ell-1}<|z|\le2^{\ell}}\frac{|f(z)|}{|z|^n}dz\bigg).
\end{align}
This estimate together with (3.3) implies
\begin{align}
\big|\mathcal S_\beta(f\chi_{\ell})(x)\big|&\le C\cdot\frac{1}{|B_{\ell}|}\bigg(\int_{2^{\ell-1}<|z|\le2^{\ell}}|f(z)|\,dz\bigg)\notag\\
&\le C\cdot\big[w_2(B_{\ell})\big]^{-1/q}\big\|f\chi_{\ell}\big\|_{L^q_{w_2}}.
\end{align}
Hence
\begin{equation*}
\begin{split}
I_3&\le C\sum_{k\in\mathbb Z}\big[w_1(B_k)\big]^{{\alpha p}/n}\left(\sum_{\ell={k+2}}^{\infty}\bigg\{\int_{2^{k-1}<|x|\le 2^k}\big|\mathcal S_\beta(f\chi_{\ell})(x)\big|^qw_2(x)\,dx\bigg\}^{1/q}\right)^p\\
&\le C\sum_{k\in\mathbb Z}\big[w_1(B_k)\big]^{{\alpha p}/n}\left(\sum_{\ell={k+2}}^{\infty}\big[w_2(B_{\ell})\big]^{-1/q}\big\|f\chi_{\ell}\big\|_{L^q_{w_2}}
\bigg\{\int_{2^{k-1}<|x|\le 2^k}w_2(x)\,dx\bigg\}^{1/q}\right)^p\\
&\le C\sum_{k\in\mathbb Z}\big[w_1(B_k)\big]^{{\alpha p}/n}\left(\sum_{\ell={k+2}}^{\infty}\frac{[w_2(B_{k})]^{1/q}}{[w_2(B_{\ell})]^{1/q}}
\big\|f\chi_{\ell}\big\|_{L^q_{w_2}}\right)^p.
\end{split}
\end{equation*}
Now we will consider the following two cases again. For the case of $0<p\le1$, by using the inequality $\left(\sum_{\ell}|a_\ell|\right)^p\le\sum_{\ell}|a_\ell|^p$ and changing the order of summation, we obtain
\begin{equation*}
\begin{split}
I_3&\le C\sum_{\ell\in\mathbb Z}\big[w_1(B_{\ell})\big]^{{\alpha p}/n}\big\|f\chi_{\ell}\big\|^p_{L^q_{w_2}}
\left(\sum_{k=-\infty}^{\ell-2}\frac{[w_2(B_k)]^{p/q}}{[w_2(B_{\ell})]^{p/q}}\cdot\frac{[w_1(B_k)]^{{\alpha p}/n}}{[w_1(B_{\ell})]^{{\alpha p}/n}}\right)\\
&\le C\sum_{\ell\in\mathbb Z}\big[w_1(B_{\ell})\big]^{{\alpha p}/n}\big\|f\chi_{\ell}\big\|^p_{L^q_{w_2}}
\left(\sum_{k=-\infty}^{\ell-2}\frac{[w_2(B_k)]^{p/q}}{[w_2(B_{\ell-2})]^{p/q}}\cdot\frac{[w_1(B_k)]^{{\alpha p}/n}}{[w_1(B_{\ell-2})]^{{\alpha p}/n}}\right).
\end{split}
\end{equation*}
Since $w_i\in A_{q_i}$, then there exist $r_i>1$ such that $w_i\in RH_{r_i}$ for $i=1, 2$. Thus by Lemma 2.2 again, we can get
\begin{equation}
\frac{w_i(B_k)}{w_i(B_{\ell-2})}\le C\left(\frac{|B_k|}{|B_{\ell-2}|}\right)^{\delta_i},\quad \mbox{for $i=1$ and 2},
\end{equation}
where $\delta_i={(r_i-1)}/{r_i}>0$. Therefore, we have
\begin{equation*}
\begin{split}
I_3&\le C\sum_{\ell\in\mathbb Z}\big[w_1(B_{\ell})\big]^{{\alpha p}/n}\big\|f\chi_{\ell}\big\|^p_{L^q_{w_2}}
\left(\sum_{k=-\infty}^{\ell-2}\left[\frac{|B_k|}{|B_{\ell-2}|}\right]^{{\alpha\delta_1p}/n+{\delta_2p}/q}\right)\\
&\le C\sum_{\ell\in\mathbb Z}\big[w_1(B_{\ell})\big]^{{\alpha p}/n}\big\|f\chi_{\ell}\big\|^p_{L^q_{w_2}}
\left(\sum_{k=-\infty}^0 2^{kn({\alpha\delta_1p}/n+{\delta_2p}/q)}\right)\\
&\le C\sum_{\ell\in\mathbb Z}\big[w_1(B_{\ell})\big]^{{\alpha p}/n}\big\|f\chi_{\ell}\big\|^p_{L^q_{w_2}},
\end{split}
\end{equation*}
where in the last inequality we have used the fact that ${\alpha\delta_1p}/n+{\delta_2p}/q>0$ under our assumption $(i)$ or $(ii)$. On the other hand, for the case of $1<p<\infty$, an application of H\"older's inequality gives us that
\begin{equation*}
\begin{split}
&\left(\sum_{\ell={k+2}}^{\infty}\frac{[w_2(B_{k})]^{1/q}}{[w_2(B_{\ell})]^{1/q}}\cdot\big[w_1(B_k)\big]^{\alpha/n}
\big\|f\chi_{\ell}\big\|_{L^q_{w_2}}\right)^p\\
\le &\left(\sum_{\ell={k+2}}^{\infty}\big[w_1(B_{\ell})\big]^{{\alpha p}/n}\big\|f\chi_{\ell}\big\|^p_{L^q_{w_2}}\cdot\frac{[w_2(B_k)]^{p/{2q}}}{[w_2(B_{\ell})]^{p/{2q}}}
\cdot\frac{[w_1(B_k)]^{{\alpha p}/{2n}}}{[w_1(B_{\ell})]^{{\alpha p}/{2n}}}\right)\\
&\times\left(\sum_{\ell={k+2}}^{\infty}\frac{[w_2(B_k)]^{{p'}/{2q}}}{[w_2(B_{\ell})]^{{p'}/{2q}}}
\cdot\frac{[w_1(B_k)]^{{\alpha p'}/{2n}}}{[w_1(B_{\ell})]^{{\alpha p'}/{2n}}}\right)^{p/{p'}}.
\end{split}
\end{equation*}
By using the same arguments as for $I_3$, we are able to prove that the following two series is bounded by an absolute constant under the assumption $(i)$ or $(ii)$.
\begin{equation}
\sum_{k=-\infty}^{\ell-2}\frac{[w_2(B_k)]^{p/{2q}}}{[w_2(B_{\ell})]^{p/{2q}}}
\cdot\frac{[w_1(B_k)]^{{\alpha p}/{2n}}}{[w_1(B_{\ell})]^{{\alpha p}/{2n}}}\le C
\end{equation}
and
\begin{equation}
\sum_{\ell={k+2}}^{\infty}\frac{[w_2(B_k)]^{{p'}/{2q}}}{[w_2(B_{\ell})]^{{p'}/{2q}}}\cdot\frac{[w_1(B_k)]^{{\alpha p'}/{2n}}}{[w_1(B_{\ell})]^{{\alpha p'}/{2n}}}\le C.
\end{equation}
Consequently
\begin{equation*}
\begin{split}
I_3&\le C\sum_{k\in\mathbb Z}\left(\sum_{\ell={k+2}}^{\infty}\big[w_1(B_{\ell})\big]^{{\alpha p}/n}\big\|f\chi_{\ell}\big\|^p_{L^q_{w_2}}\cdot\frac{[w_2(B_k)]^{p/{2q}}}{[w_2(B_{\ell})]^{p/{2q}}}
\cdot\frac{[w_1(B_k)]^{{\alpha p}/{2n}}}{[w_1(B_{\ell})]^{{\alpha p}/{2n}}}\right)\\
&\le C\sum_{\ell\in\mathbb Z}\big[w_1(B_{\ell})\big]^{{\alpha p}/n}\big\|f\chi_{\ell}\big\|^p_{L^q_{w_2}}
\left(\sum_{k=-\infty}^{\ell-2}\frac{[w_2(B_k)]^{p/{2q}}}{[w_2(B_{\ell})]^{p/{2q}}}
\cdot\frac{[w_1(B_k)]^{{\alpha p}/{2n}}}{[w_1(B_{\ell})]^{{\alpha p}/{2n}}}\right)\\
&\le C\sum_{\ell\in\mathbb Z}\big[w_1(B_{\ell})\big]^{{\alpha p}/n}\big\|f\chi_{\ell}\big\|^p_{L^q_{w_2}}.
\end{split}
\end{equation*}
From the above discussions for the term $I_3$, we know that for any $0<p<\infty$,
\begin{equation*}
I_3\le C\sum_{\ell\in\mathbb Z}\big[w_1(B_{\ell})\big]^{{\alpha p}/n}\big\|f\chi_{\ell}\big\|^p_{L^q_{w_2}}\le C\big\|f\big\|^p_{\dot K^{\alpha,p}_q(w_1,w_2)}.
\end{equation*}
Summing up the above estimates for $I_1$, $I_2$ and $I_3$, we complete the proof of Theorem 1.1.
\end{proof}

\begin{proof}[Proof of Theorem 1.2]
Let $f\in \dot K^{\alpha,p}_q(w_1,w_2)$. For any $k\in\mathbb Z$, as in the proof of Theorem 1.1, we will split $f(x)$ into three parts
\begin{equation*}
\begin{split}
f(x)&=f(x)\chi_{\{2^{k-2}<|x|\le 2^{k+1}\}}(x)+f(x)\chi_{\{|x|\le 2^{k-2}\}}(x)+f(x)\chi_{\{|x|>2^{k+1}\}}(x)\\
&=f_1(x)+f_2(x)+f_3(x).
\end{split}
\end{equation*}
Then for any given $\lambda>0$, we have
\begin{equation*}
\begin{split}
&\lambda^p\cdot\sum_{k\in\mathbb Z}\big[w_1(B_k)\big]^{{\alpha p}/n}w_2\Big(\Big\{x\in C_k:|\mathcal S_\beta(f)(x)|>\lambda\Big\}\Big)^{p/q}\\
\le\,&\sum_{i=1}^3\lambda^p\cdot\sum_{k\in\mathbb Z}\big[w_1(B_k)\big]^{{\alpha p}/n}
w_2\Big(\Big\{x\in C_k:|\mathcal S_\beta(f_i)(x)|>\lambda/3\Big\}\Big)^{p/q}\\
=\,&I'_1+I'_2+I'_3.
\end{split}
\end{equation*}
Applying Chebyshev's inequality, Theorem A and Lemma 2.1, we obtain
\begin{equation*}
\begin{split}
I'_1&\le\lambda^p\cdot\sum_{k\in\mathbb Z}\big[w_1(B_k)\big]^{{\alpha p}/n}\left(\frac{3^q}{\lambda^q}\big\|\mathcal S_\beta(f_1)\big\|^q_{L^q_{w_2}}\right)^{p/q}\\
&\le C\sum_{k\in\mathbb Z}\big[w_1(B_k)\big]^{{\alpha p}/n}\big\|f_1\big\|_{L^q_{w_2}}^p\\
&\le C\sum_{k\in\mathbb Z}\big[w_1(B_k)\big]^{{\alpha p}/n}\big\|f\chi_k\big\|_{L^q_{w_2}}^p\\
&\le C\big\|f\big\|^p_{\dot K^{\alpha,p}_q(w_1,w_2)}.
\end{split}
\end{equation*}
For any $x\in C_k$, it follows from the inequalities (3.2) and (3.3) that
\begin{equation*}
\begin{split}
\big|\mathcal S_\beta(f_2)(x)\big|&\le \sum_{\ell=-\infty}^{k-2}\big|\mathcal S_\beta(f\chi_{\ell})(x)\big|\\
&\le C \sum_{\ell=-\infty}^{k-2}\frac{1}{|x|^n}\bigg(\int_{2^{\ell-1}<|z|\le2^{\ell}}|f(z)|\,dz\bigg)\\
&\le C \sum_{\ell=-\infty}^{k-2}\frac{|B_{\ell}|}{|B_k|}
\big[w_2(B_{\ell})\big]^{-1/q}\big\|f\chi_{\ell}\big\|_{L^q_{w_2}}.
\end{split}
\end{equation*}
By using Lemma 2.1, the inequality (3.4) and the fact that $\alpha q_1=n(1-{q_2}/q)$, we deduce that
\begin{equation*}
\begin{split}
\big|\mathcal S_\beta(f_2)(x)\big|\le&\, C\cdot\frac{1}{[w_1(B_k)]^{\alpha/n}[w_2(B_k)]^{1/q}}\\
&\times\sum_{\ell=-\infty}^{k-2}\big[w_1(B_{\ell})\big]^{\alpha/n}
\big\|f\chi_{\ell}\big\|_{L^q_{w_2}}\cdot\frac{|B_{\ell}|}{|B_k|}
\cdot\frac{[w_2(B_k)]^{1/q}}{[w_2(B_{\ell+2})]^{1/q}}
\cdot\frac{[w_1(B_k)]^{\alpha/n}}{[w_1(B_{\ell+2})]^{\alpha/n}}\\
\le&\, C\cdot\frac{1}{[w_1(B_k)]^{\alpha/n}[w_2(B_k)]^{1/q}}\\
&\times\sum_{\ell=-\infty}^{k-2}\big[w_1(B_{\ell})\big]^{\alpha/n}
\big\|f\chi_{\ell}\big\|_{L^q_{w_2}}\cdot\left(\frac{|B_{\ell+2}|}{|B_k|}\right)^{1-{\alpha q_1}/n-{q_2}/q}\\
\le&\, C\cdot\frac{1}{[w_1(B_k)]^{\alpha/n}[w_2(B_k)]^{1/q}}
\sum_{\ell=-\infty}^{k-2}\big[w_1(B_{\ell})\big]^{{\alpha}/n}\big\|f\chi_{\ell}\big\|_{L^q_{w_2}}.
\end{split}
\end{equation*}
Moreover, since $0<p\le1$, then we have that for any $x\in C_k$,
\begin{align}
\big|\mathcal S_\beta(f_2)(x)\big|&\le C\cdot\frac{1}{[w_1(B_k)]^{\alpha/n}[w_2(B_k)]^{1/q}}
\left(\sum_{\ell=-\infty}^{k-2}\big[w_1(B_{\ell})\big]^{{\alpha p}/n}\big\|f\chi_{\ell}\big\|^p_{L^q_{w_2}}\right)^{1/p}\notag\\
&\le C\cdot\frac{1}{[w_1(B_k)]^{\alpha/n}[w_2(B_k)]^{1/q}}\big\|f\big\|_{\dot K^{\alpha,p}_q(w_1,w_2)}.
\end{align}
Set $A_k=[w_1(B_k)]^{-\alpha/n}[w_2(B_k)]^{-1/q}$. If $\big\{x\in C_k:|\mathcal S_\beta(f_2)(x)|>\lambda/3\big\}=\O$, then the inequality
\begin{equation*}
I'_2\le C\big\|f\big\|^p_{\dot K^{\alpha,p}_q(w_1,w_2)}
\end{equation*}
holds trivially. Now we suppose that $\big\{x\in C_k:|\mathcal S_\beta(f_2)(x)|>\lambda/3\big\}\neq\O$. First it is easy to verify that $\lim_{k\to\infty}A_k=0$. Then for any fixed $\lambda>0$, we are able to find a maximal positive integer $k_\lambda$ such that
\begin{equation}
\lambda/3\le C\cdot A_{k_\lambda}\big\|f\big\|_{\dot K^{\alpha,p}_q(w_1,w_2)}.
\end{equation}
Hence
\begin{equation*}
\begin{split}
I'_2&\le\lambda^p\sum_{k=-\infty}^{k_\lambda}[w_1(B_k)]^{{\alpha p}/n}[w_2(B_k)]^{p/q}\\
&\le C\big\|f\big\|^p_{\dot K^{\alpha,p}_q(w_1,w_2)}\sum_{k=-\infty}^{k_\lambda}
\frac{[w_1(B_k)]^{{\alpha p}/n}}{[w_1(B_{k_\lambda})]^{{\alpha p}/n}}\cdot\frac{[w_2(B_k)]^{p/q}}{[w_2(B_{k_\lambda})]^{p/q}}.
\end{split}
\end{equation*}
Because $B_k\subseteq B_{k_\lambda}$, then by Lemma 2.2 with the same notations $\delta_i$ as in (3.9), we can get
\begin{equation*}
\frac{w_i(B_k)}{w_i(B_{k_\lambda})}\le C\left(\frac{|B_k|}{|B_{k_\lambda}|}\right)^{\delta_i},\quad \mbox{for $i=1$ and 2}.
\end{equation*}
Therefore
\begin{equation*}
\begin{split}
I'_2&\le C\big\|f\big\|^p_{\dot K^{\alpha,p}_q(w_1,w_2)}\sum_{k=-\infty}^{k_\lambda}
\left(\frac{|B_k|}{|B_{k_\lambda}|}\right)^{{\alpha\delta_1p}/n+{\delta_2p}/q}\\
&\le C\big\|f\big\|^p_{\dot K^{\alpha,p}_q(w_1,w_2)}.
\end{split}
\end{equation*}
On the other hand, it follows from the inequalities (3.3) and (3.7) that
\begin{equation*}
\begin{split}
\big|\mathcal S_\beta(f_3)(x)\big|&\le \sum_{\ell={k+2}}^{\infty}\big|\mathcal S_\beta(f\chi_{\ell})(x)\big|\\
&\le C\sum_{\ell={k+2}}^{\infty}\int_{2^{\ell-1}<|z|\le2^{\ell}}\frac{|f(z)|}{|z|^n}dz\\
&\le C\sum_{\ell={k+2}}^{\infty}\big[w_2(B_{\ell})\big]^{-1/q}\big\|f\chi_{\ell}\big\|_{L^q_{w_2}}.
\end{split}
\end{equation*}
In the present situation, since $B_k\subseteq B_{\ell-2}$ with $\ell\ge k+2$, then it follows from the inequality (3.9) that
\begin{equation*}
\begin{split}
\big|\mathcal S_\beta(f_3)(x)\big|\le&\, C\cdot\frac{1}{[w_1(B_k)]^{\alpha/n}[w_2(B_k)]^{1/q}}\\
&\times\sum_{\ell={k+2}}^{\infty}
\big[w_1(B_{\ell})\big]^{\alpha/n}\big\|f\chi_{\ell}\big\|_{L^q_{w_2}}
\cdot\frac{[w_2(B_k)]^{1/q}}{[w_2(B_{\ell-2})]^{1/q}}
\cdot\frac{[w_1(B_k)]^{\alpha/n}}{[w_1(B_{\ell-2})]^{\alpha/n}}\\
\le&\, C\cdot\frac{1}{[w_1(B_k)]^{\alpha/n}[w_2(B_k)]^{1/q}}\\
&\times\sum_{\ell={k+2}}^{\infty}
\big[w_1(B_{\ell})\big]^{\alpha/n}\big\|f\chi_{\ell}\big\|_{L^q_{w_2}}
\cdot\left(\frac{|B_k|}{|B_{\ell-2}|}\right)^{{\alpha\delta_1}/n+{\delta_2}/q}\\
\le&\, C\cdot\frac{1}{[w_1(B_k)]^{\alpha/n}[w_2(B_k)]^{1/q}}\sum_{\ell={k+2}}^{\infty}
\big[w_1(B_{\ell})\big]^{\alpha/n}\big\|f\chi_{\ell}\big\|_{L^q_{w_2}}.
\end{split}
\end{equation*}
Furthermore, recall that $0<p\le1$, then for any $x\in C_k$, we have
\begin{equation*}
\begin{split}
\big|\mathcal S_\beta(f_3)(x)\big|&\le C\cdot\frac{1}{[w_1(B_k)]^{\alpha/n}[w_2(B_k)]^{1/q}}
\left(\sum_{\ell={k+2}}^{\infty}\big[w_1(B_{\ell})\big]^{{\alpha p}/n}\big\|f\chi_{\ell}\big\|^p_{L^q_{w_2}}\right)^{1/p}\\
&\le C\cdot\frac{1}{[w_1(B_k)]^{\alpha/n}[w_2(B_k)]^{1/q}}\big\|f\big\|_{\dot K^{\alpha,p}_q(w_1,w_2)}.
\end{split}
\end{equation*}
Repeating the arguments used for the term $I'_2$, we can also obtain
\begin{equation*}
I'_3\le C\big\|f\big\|^p_{\dot K^{\alpha,p}_q(w_1,w_2)}.
\end{equation*}
Combining the above estimates for $I'_1$, $I'_2$ and $I'_3$, and then taking the supremum over all $\lambda>0$, we finish the proof of Theorem 1.2.
\end{proof}

\section{Proofs of Theorems 1.3 and 1.4}

In order to prove the main theorems of this section, let us first establish the following results.

\newtheorem{prop}[theorem]{Proposition}

\begin{prop}
Let $0<\beta\le1$, $q=2$ and $w\in A_{q_2}$ with $1\le q_2\le q$. Then for any $j\in\mathbb Z_+$, we have
\begin{equation*}
\big\|\mathcal S_{\beta,2^j}(f)\big\|_{L^2_{w}}\le C\cdot2^{{jnq_2}/2}\big\|\mathcal S_\beta(f)\big\|_{L^2_{w}}.
\end{equation*}
\end{prop}

\begin{proof}
Since $w\in A_{q_2}$, then by Lemma 2.1, we know that for any $(y,t)\in{\mathbb R}^{n+1}_+$,
\begin{equation*}
w\big(B(y,2^jt)\big)=w\big(2^jB(y,t)\big)\le C\cdot2^{jnq_2}w\big(B(y,t)\big) \quad j=1,2,\ldots.
\end{equation*}
Therefore
\begin{equation*}
\begin{split}
\big\|\mathcal S_{\beta,2^j}(f)\big\|_{L^2_{w}}^2&=\int_{\mathbb R^n}\bigg(\iint_{{\mathbb R}^{n+1}_+}\Big(A_\beta(f)(y,t)\Big)^2\chi_{|x-y|<2^j t}\frac{dydt}{t^{n+1}}\bigg)w(x)\,dx\\
&=\iint_{{\mathbb R}^{n+1}_+}\bigg(\int_{|x-y|<2^j t}w(x)\,dx\bigg)\Big(A_\beta(f)(y,t)\Big)^2\frac{dydt}{t^{n+1}}\\
&\le C\cdot2^{jnq_2}\iint_{{\mathbb R}^{n+1}_+}\bigg(\int_{|x-y|<t}w(x)\,dx\bigg)\Big(A_\beta(f)(y,t)\Big)^2\frac{dydt}{t^{n+1}}\\
&=C\cdot 2^{jnq_2}\big\|\mathcal S_\beta(f)\big\|_{L^2_{w}}^2.
\end{split}
\end{equation*}
Taking square-roots on both sides of the above inequality, we are done.
\end{proof}

\begin{prop}
Let $0<\beta\le1$, $2<q<\infty$ and $w\in A_{q_2}$ with $1\le q_2\le q$. Then for any $j\in\mathbb Z_+$, we have
\begin{equation*}
\big\|\mathcal S_{\beta,2^j}(f)\big\|_{L^q_{w}}\le C\cdot2^{{jnq_2}/2}\big\|\mathcal S_\beta(f)\big\|_{L^q_{w}}.
\end{equation*}
\end{prop}

\begin{proof}
For any $j\in\mathbb Z_+$ and $0<\beta\le1$, it is easy to see that
\begin{equation}
\big\|\mathcal S_{\beta,2^j}(f)\big\|^2_{L^q_w}=\big\|\mathcal S_{\beta,2^j}(f)^2\big\|_{L^{q/2}_w}.
\end{equation}
Since $q/2>1$, then by duality, we have
\begin{align}
&\big\|\mathcal S_{\beta,2^j}(f)^2\big\|_{L^{q/2}_w}\notag\\
=&\sup_{\|b\|_{L_w^{(q/2)'}}\le1}\left|\int_{\mathbb R^n}\mathcal S_{\beta,2^j}(f)(x)^2b(x)w(x)\,dx\right|\notag\\
=&\sup_{\|b\|_{L_w^{(q/2)'}}\le1}\left|\int_{\mathbb R^n}\bigg(\iint_{{\mathbb R}^{n+1}_+}\Big(A_\beta(f)(y,t)\Big)^2\chi_{|x-y|<2^j t}\frac{dydt}{t^{n+1}}\bigg)b(x)w(x)\,dx\right|\notag\\
=&\sup_{\|b\|_{L_w^{(q/2)'}}\le1}\left|\iint_{{\mathbb R}^{n+1}_+}\bigg(\int_{|x-y|<2^jt}b(x)w(x)\,dx\bigg)\Big(A_\beta(f)(y,t)\Big)^2 \frac{dydt}{t^{n+1}}\right|.
\end{align}
For $w\in A_{q_2}$, we denote the weighted maximal operator by $M_w$; that is
\begin{equation*}
M_w(f)(x)=\underset{x\in B}{\sup}\frac{1}{w(B)}\int_B|f(y)|w(y)\,dy,
\end{equation*}
where the supremum is taken over all balls $B$ which contain $x$. Then, by Lemma 2.1, we can get
\begin{align}
\int_{|x-y|<2^jt}b(x)w(x)\,dx&\le C\cdot2^{jnq_2}w\big(B(y,t)\big)\cdot\frac{1}{w(B(y,2^jt))}\int_{B(y,2^jt)}b(x)w(x)\,dx\notag\\
&\le C\cdot2^{jnq_2}w\big(B(y,t)\big)\underset{x\in B(y,2^jt)}{\inf}M_w(b)(x)\notag\\
&\le C\cdot2^{jnq_2}\int_{|x-y|<t}M_w(b)(x)w(x)\,dx.
\end{align}
Substituting the above inequality (4.3) into (4.2) and then using H\"older's inequality together with the $L^{(q/2)'}_w$ boundedness of $M_w$, we thus obtain
\begin{equation*}
\begin{split}
\big\|\mathcal S_{\beta,2^j}(f)^2\big\|_{L^{q/2}_w}&\le C\cdot2^{jnq_2}
\sup_{\|b\|_{L_w^{(q/2)'}}\le1}\left|\int_{\mathbb R^n}\mathcal S_\beta(f)(x)^2M_w(b)(x)w(x)\,dx\right|\\
&\le C\cdot2^{jnq_2}\big\|\mathcal S_\beta(f)^2\big\|_{L^{q/2}_w}
\sup_{\|b\|_{L_w^{(q/2)'}}\le1}\big\|M_w(b)\big\|_{L^{(q/2)'}_w}\\
&\le C\cdot2^{jnq_2}\big\|\mathcal S_\beta(f)^2\big\|_{L^{q/2}_w}\\
&  = C\cdot2^{jnq_2}\big\|\mathcal S_\beta(f)\big\|^2_{L^q_w}.
\end{split}
\end{equation*}
This estimate together with (4.1) implies the desired result.
\end{proof}

\begin{prop}
Let $0<\beta\le1$, $1<q<2$ and $w\in A_{q_2}$ with $1\le q_2\le q$. Then for any $j\in\mathbb Z_+$, we have
\begin{equation*}
\big\|\mathcal S_{\beta,2^j}(f)\big\|_{L^q_w}\le C\cdot2^{{jnq_2}/q}\big\|\mathcal S_\beta(f)\big\|_{L^q_w}.
\end{equation*}
\end{prop}

\begin{proof}
We will adopt the same method given in \cite{torchinsky}. For any $j\in\mathbb Z_+$, set $\Omega_\lambda=\big\{x\in\mathbb R^n:\mathcal S_\beta(f)(x)>\lambda\big\}$ and $\Omega_{\lambda,j}=\big\{x\in\mathbb R^n:\mathcal S_{\beta,2^j}(f)(x)>\lambda\big\}.$ We also set
\begin{equation*}
\Omega^*_\lambda=\Big\{x\in\mathbb R^n:M_w(\chi_{\Omega_\lambda})(x)>\frac{1}{2^{(jnq_2+1)}\cdot[w]_{A_{q_2}}}\Big\}.
\end{equation*}
Observe that $w\big(\Omega_{\lambda,j}\big)\le w\big(\Omega^*_\lambda\big)+w\big(\Omega_{\lambda,j}\cap(\mathbb R^n\backslash\Omega^*_\lambda)\big)$. Thus, for any $j\in\mathbb Z_+$,
\begin{equation*}
\begin{split}
\big\|\mathcal S_{\beta,2^j}(f)\big\|^q_{L^q_w}&=\int_0^\infty q\lambda^{q-1}w\big(\Omega_{\lambda,j}\big)\,d\lambda\\
&\le\int_0^\infty q\lambda^{q-1}w\big(\Omega^*_\lambda\big)\,d\lambda+\int_0^\infty q\lambda^{q-1}w\big(\Omega_{\lambda,j}\cap(\mathbb R^n\backslash\Omega^*_\lambda)\big)\,d\lambda\\
&=\mbox{\upshape I+II}.
\end{split}
\end{equation*}
The weighted weak type estimate of $M_w$ yields
\begin{equation}
\mbox{\upshape I}\le C\cdot2^{jnq_2}\int_0^\infty q\lambda^{q-1}w(\Omega_\lambda)\,d\lambda\le C\cdot2^{jnq_2}\big\|\mathcal S_\beta(f)\big\|^q_{L^q_w}.
\end{equation}
To estimate II, we now claim that the following inequality holds.
\begin{equation}
\int_{\mathbb R^n\backslash\Omega^*_\lambda}\mathcal S_{\beta,2^j}(f)(x)^2w(x)\,dx\le C\cdot2^{jnq_2}\int_{\mathbb R^n\backslash\Omega_\lambda}\mathcal S_{\beta}(f)(x)^2w(x)\,dx.
\end{equation}
Assuming the claim for the moment, then it follows from Chebyshev's inequality and the inequality (4.5) that
\begin{equation*}
\begin{split}
w\big(\Omega_{\lambda,j}\cap(\mathbb R^n\backslash\Omega^*_\lambda)\big)&\le\lambda^{-2}\int_{\Omega_{\lambda,j}\cap(\mathbb R^n\backslash\Omega^*_\lambda)}\mathcal S_{\beta,2^j}(f)(x)^2w(x)\,dx\\
&\le\lambda^{-2}\int_{\mathbb R^n\backslash\Omega^*_\lambda}\mathcal S_{\beta,2^j}(f)(x)^2w(x)\,dx\\
&\le C\cdot2^{jnq_2}\lambda^{-2}\int_{\mathbb R^n\backslash\Omega_\lambda}\mathcal S_{\beta}(f)(x)^2w(x)\,dx.
\end{split}
\end{equation*}
Hence
\begin{equation*}
\mbox{\upshape II}\le C\cdot2^{jnq_2}\int_0^\infty q\lambda^{q-1}\bigg(\lambda^{-2}\int_{\mathbb R^n\backslash\Omega_\lambda}\mathcal S_{\beta}(f)(x)^2w(x)\,dx\bigg)d\lambda.
\end{equation*}
Changing the order of integration yields
\begin{align}
\mbox{\upshape II}&\le C\cdot2^{jnq_2}\int_{\mathbb R^n}\mathcal S_\beta(f)(x)^2\bigg(\int_{|\mathcal S_\beta(f)(x)|}^\infty q\lambda^{q-3}\,d\lambda\bigg)w(x)\,dx\notag\\
&\le C\cdot2^{jnq_2}\frac{q}{2-q}\cdot\big\|\mathcal S_\beta(f)\big\|^q_{L^q_w}.
\end{align}
Combining the above estimate (4.6) with (4.4) and taking $q$-th root on both sides, we are done. So it remains to prove the inequality (4.5). Set $\Gamma_{2^j}(\mathbb R^n\backslash\Omega^*_\lambda)=\underset{x\in\mathbb R^n\backslash\Omega^*_\lambda}{\bigcup}\Gamma_{2^j}(x)$ and
$\Gamma(\mathbb R^n\backslash\Omega_\lambda)=\underset{x\in\mathbb R^n\backslash\Omega_\lambda}{\bigcup}\Gamma(x).$
For each given $(y,t)\in\Gamma_{2^j}(\mathbb R^n\backslash\Omega^*_\lambda)$, by Lemma 2.1, we thus have
\begin{equation*}
w\big(B(y,2^jt)\cap(\mathbb R^n\backslash\Omega_\lambda^*)\big)\le C\cdot2^{jnq_2}w\big(B(y,t)\big).
\end{equation*}
It is not difficult to check that $w\big(B(y,t)\cap\Omega_\lambda\big)\le\frac{w(B(y,t))}{2}$ and $\Gamma_{2^j}(\mathbb R^n\backslash\Omega^*_\lambda)\subseteq\Gamma(\mathbb R^n\backslash\Omega_\lambda)$. In fact, for any $(y,t)\in\Gamma_{2^j}(\mathbb R^n\backslash\Omega^*_\lambda)$, there exists a point $x\in \mathbb R^n\backslash\Omega^*_\lambda$ such that $(y,t)\in\Gamma_{2^j}(x)$. Then we can deduce
\begin{equation*}
\begin{split}
w\big(B(y,t)\cap\Omega_\lambda\big)&\le w\big(B(y,2^jt)\cap\Omega_\lambda\big)\\
&= \int_{B(y,2^jt)}\chi_{\Omega_\lambda}(z)w(z)\,dz\\
&\le [w]_{A_{q_2}}\cdot2^{jnq_2}w\big(B(y,t)\big)\cdot
\frac{1}{w(B(y,2^jt))}\int_{B(y,2^jt)}\chi_{\Omega_\lambda}(z)w(z)\,dz.
\end{split}
\end{equation*}
Note that $x\in B(y,2^jt)\cap(\mathbb R^n\backslash\Omega^*_\lambda)$. So we have
\begin{equation*}
\begin{split}
w\big(B(y,t)\cap\Omega_\lambda\big)\le [w]_{A_{q_2}}\cdot2^{jnq_2}w\big(B(y,t)\big)
M_w(\chi_{\Omega_\lambda})(x)\le \frac{w(B(y,t))}{2}.
\end{split}
\end{equation*}
Hence
\begin{equation*}
\begin{split}
w\big(B(y,t)\big)&=w\big(B(y,t)\cap\Omega_\lambda\big)+w\big(B(y,t)\cap(\mathbb R^n\backslash\Omega_\lambda)\big)\\
&\le \frac{w(B(y,t))}{2}+w\big(B(y,t)\cap(\mathbb R^n\backslash\Omega_\lambda)\big),
\end{split}
\end{equation*}
which is equivalent to
\begin{equation*}
w\big(B(y,t)\big)\le 2\cdot w\big(B(y,t)\cap(\mathbb R^n\backslash\Omega_\lambda)\big).
\end{equation*}
The above inequality implies in particular that there is a point $z\in B(y,t)\cap(\mathbb R^n\backslash\Omega_\lambda)\neq\emptyset$. In this case, we have $(y,t)\in\Gamma(z)$ with $z\in \mathbb R^n\backslash\Omega_\lambda$, which implies $\Gamma_{2^j}(\mathbb R^n\backslash\Omega^*_\lambda)\subseteq\Gamma(\mathbb R^n\backslash\Omega_\lambda)$. Thus we obtain
\begin{equation*}
w\big(B(y,2^jt)\cap(\mathbb R^n\backslash\Omega_\lambda^*)\big)\le C\cdot2^{jnq_2}w\big(B(y,t)\cap(\mathbb R^n\backslash\Omega_\lambda)\big).
\end{equation*}
Therefore
\begin{equation*}
\begin{split}
&\int_{\mathbb R^n\backslash\Omega^*_\lambda}\mathcal S_{\beta,2^j}(f)(x)^2w(x)\,dx\\
=&\int_{\mathbb R^n\backslash\Omega^*_\lambda}\bigg(\iint_{\Gamma_{2^j}(x)}\Big(A_\beta(f)(y,t)\Big)^2\frac{dydt}{t^{n+1}}
\bigg)w(x)\,dx\\
\le&\iint_{\Gamma_{2^j}(\mathbb R^n\backslash\Omega^*_\lambda)}\bigg(\int_{B(y,2^jt)\cap(\mathbb R^n\backslash\Omega_\lambda^*)}w(x)\,dx\bigg)\Big(A_\beta(f)(y,t)\Big)^2\frac{dydt}{t^{n+1}}\\
\le&\,C\cdot2^{jnq_2}\iint_{\Gamma(\mathbb R^n\backslash\Omega_\lambda)}\bigg(\int_{B(y,t)\cap(\mathbb R^n\backslash\Omega_{\lambda})}w(x)\,dx\bigg)\Big(A_\beta(f)(y,t)\Big)^2\frac{dydt}{t^{n+1}}\\
\le&\,C\cdot2^{jnq_2}\int_{\mathbb R^n\backslash\Omega_\lambda}\mathcal S_{\beta}(f)(x)^2w(x)\,dx,
\end{split}
\end{equation*}
which is exactly what we want. This completes the proof of Proposition 4.3.
\end{proof}

We are now in a position to give the proofs of the main theorems.

\begin{proof}[Proof of Theorem 1.3]
From the definition of $\mathcal G^*_{\lambda,\beta}$, we readily see that
\begin{align}
\left|\mathcal G^*_{\lambda,\beta}(f)(x)\right|^2=&\iint_{\mathbb R^{n+1}_+}\left(\frac{t}{t+|x-y|}\right)^{\lambda n}\Big(A_\beta(f)(y,t)\Big)^2\frac{dydt}{t^{n+1}}\notag\\
=&\int_0^\infty\int_{|x-y|<t}\left(\frac{t}{t+|x-y|}\right)^{\lambda n}\Big(A_\beta(f)(y,t)\Big)^2\frac{dydt}{t^{n+1}}\notag\\
&+\sum_{j=1}^\infty\int_0^\infty\int_{2^{j-1}t\le|x-y|<2^jt}\left(\frac{t}{t+|x-y|}\right)^{\lambda n}\Big(A_\beta(f)(y,t)\Big)^2\frac{dydt}{t^{n+1}}\notag\\
\le&\, C\bigg[\mathcal S_\beta(f)(x)^2+\sum_{j=1}^\infty 2^{-j\lambda n}\mathcal S_{\beta,2^j}(f)(x)^2\bigg].
\end{align}
Let $f\in \dot K^{\alpha,p}_q(w_1,w_2)$. We decompose $f(x)=f_1(x)+f_2(x)+f_3(x)$ as in Theorem 1.1, then we have
\begin{equation*}
\begin{split}
\big\|\mathcal G^*_{\lambda,\beta}(f)\big\|^p_{\dot K^{\alpha,p}_q(w_1,w_2)}
&\le C\sum_{i=1}^3\sum_{k\in\mathbb Z}\big[w_1(B_k)\big]^{{\alpha p}/n}\big\|\mathcal G^*_{\lambda,\beta}(f_i)\chi_k\big\|_{L^q_{w_2}}^p\\
&=J_1+J_2+J_3.
\end{split}
\end{equation*}
Note that $\lambda>\max\{q_2,3\}\ge\max\{q_2,{2q_2}/q\}$ when $q_2\le q$. Since $w_2\in A_{q_2}$ and $1\le q_2\le q$, then $w_2\in A_q$. Applying Propositions 4.1--4.3, Theorem A and the above inequality (4.7), we obtain
\begin{align}
\big\|\mathcal G^*_{\lambda,\beta}(f_1)\big\|_{L^q_{w_2}}&\le C\bigg(\big\|\mathcal S_\beta(f_1)\big\|_{L^q_{w_2}}+\sum_{j=1}^\infty2^{-j\lambda n/2}\big\|\mathcal S_{\beta,2^j}(f_1)\big\|_{L^q_{w_2}}\bigg)\notag\\
&\le C\big\|f_1\big\|_{L^q_{w_2}}\bigg(1+\sum_{j=1}^\infty 2^{-j\lambda n/2}\big[2^{{jnq_2}/2}+2^{{jnq_2}/q}\big]\bigg)\notag\\
&\le C\big\|f_1\big\|_{L^q_{w_2}}.
\end{align}
From the above estimate (4.8) and Lemma 2.1, it follows that
\begin{equation*}
\begin{split}
J_1&\le C\sum_{k\in\mathbb Z}\big[w_1(B_k)\big]^{{\alpha p}/n}\big\|\mathcal G^*_{\lambda,\beta}(f_1)\big\|_{L^q_{w_2}}^p\\
&\le C\sum_{k\in\mathbb Z}\big[w_1(B_k)\big]^{{\alpha p}/n}\big\|f_1\big\|_{L^q_{w_2}}^p\\
&\le C\sum_{k\in\mathbb Z}\big[w_1(B_k)\big]^{{\alpha p}/n}\big\|f\chi_k\big\|_{L^q_{w_2}}^p\\
&\le C\big\|f\big\|^p_{\dot K^{\alpha,p}_q(w_1,w_2)}.
\end{split}
\end{equation*}
For any $j\in\mathbb Z_+$, $x\in C_k$, $(y,t)\in\Gamma_{2^j}(x)$ and $z\in\{2^{\ell-1}<|z|\le2^{\ell}\}\cap B(y,t)$ with $\ell\le k-2$, then by a simple calculation, we can easily deduce
\begin{equation*}
t+2^jt\ge |x-y|+|y-z|\ge|x-z|\ge|x|-|z|\ge \frac{|x|}{2}.
\end{equation*}
Thus, by the previous inequality (3.1) and Minkowski's inequality, we get
\begin{align}
\big|\mathcal S_{\beta,2^j}(f\chi_{\ell})(x)\big|&=\left(\iint_{\Gamma_{2^j}(x)}\Big(\sup_{\varphi\in{\mathcal C}_\beta}\big|(f\chi_{\ell})*\varphi_t(y)\big|\Big)^2\frac{dydt}{t^{n+1}}\right)^{1/2}\notag\\
&\le C\left(\int_{\frac{|x|}{2^{j+2}}}^\infty\int_{|x-y|<2^jt}\bigg|t^{-n}
\int_{2^{\ell-1}<|z|\le2^{\ell}}|f(z)|\,dz\bigg|^2\frac{dydt}{t^{n+1}}\right)^{1/2}\notag\\
&\le C\left(\int_{2^{\ell-1}<|z|\le2^{\ell}}|f(z)|\,dz\right)
\left(\int_{\frac{|x|}{2^{j+2}}}^\infty 2^{jn}\frac{dt}{t^{2n+1}}\right)^{1/2}\notag\\
&\le C\cdot2^{{3jn}/2}\frac{1}{|x|^n}\bigg(\int_{2^{\ell-1}<|z|\le2^{\ell}}|f(z)|\,dz\bigg).
\end{align}
Moreover, by using Minkowski's inequality, (3.3) and (4.9), we obtain
\begin{equation*}
\begin{split}
\big\|\mathcal S_{\beta,2^j}(f_2)\chi_k\big\|_{L^q_{w_2}}&\le\sum_{\ell=-\infty}^{k-2}\big\|\mathcal S_{\beta,2^j}(f\chi_{\ell})\chi_k\big\|_{L^q_{w_2}}\\
&\le C\cdot2^{{3jn}/2}\sum_{\ell=-\infty}^{k-2}\bigg(\int_{2^{\ell-1}<|z|\le2^{\ell}}|f(z)|\,dz\bigg)
\bigg(\int_{2^{k-1}<|x|\le 2^k}\frac{w_2(x)}{|x|^{nq}}dx\bigg)^{1/q}\\
&\le C\cdot2^{{3jn}/2}\sum_{\ell=-\infty}^{k-2}\frac{|B_{\ell}|}{|B_k|}\cdot
\frac{[w_2(B_k)]^{1/q}}{[w_2(B_{\ell})]^{1/q}}\big\|f\chi_{\ell}\big\|_{L^q_{w_2}}.
\end{split}
\end{equation*}
Consequently
\begin{equation*}
\begin{split}
J_2&\le C\sum_{k\in\mathbb Z}\big[w_1(B_k)\big]^{{\alpha p}/n}\bigg(\big\|\mathcal S_{\beta}(f_2)\chi_k\big\|_{L^q_{w_2}}+\sum_{j=1}^\infty 2^{-j\lambda n/2}\big\|\mathcal S_{\beta,2^j}(f_2)\chi_k\big\|_{L^q_{w_2}}\bigg)^p\\
&\le C\sum_{k\in\mathbb Z}\big[w_1(B_k)\big]^{{\alpha p}/n}\bigg(\sum_{\ell=-\infty}^{k-2}\frac{|B_{\ell}|}{|B_k|}\cdot
\frac{[w_2(B_k)]^{1/q}}{[w_2(B_{\ell})]^{1/q}}\big\|f\chi_{\ell}\big\|_{L^q_{w_2}}\bigg)^p
\times\bigg(1+\sum_{j=1}^\infty 2^{-j\lambda n/2}\cdot2^{{3jn}/2}\bigg)^p\\
&\le C\sum_{k\in\mathbb Z}\big[w_1(B_k)\big]^{{\alpha p}/n}\bigg(\sum_{\ell=-\infty}^{k-2}\frac{|B_{\ell}|}{|B_k|}\cdot
\frac{[w_2(B_k)]^{1/q}}{[w_2(B_{\ell})]^{1/q}}\big\|f\chi_{\ell}\big\|_{L^q_{w_2}}\bigg)^p,
\end{split}
\end{equation*}
where the last inequality holds under our assumption $\lambda>3$. On the other hand, for any $j\in\mathbb Z_+$, $x\in C_k$, $(y,t)\in\Gamma_{2^j}(x)$ and $z\in\{2^{\ell-1}<|z|\le2^{\ell}\}\cap B(y,t)$ with $\ell\ge k+2$, it is easy to verify that
\begin{equation*}
t+2^jt\ge |x-y|+|y-z|\ge|x-z|\ge|z|-|x|\ge \frac{|z|}{2}.
\end{equation*}
Then it follows from the inequality (3.1) and Minkowski's inequality that
\begin{align}
\big|\mathcal S_{\beta,2^j}(f\chi_{\ell})(x)\big|&\le C\left(\int_{\frac{|z|}{2^{j+2}}}^\infty\int_{|x-y|<2^jt}\bigg|t^{-n}
\int_{2^{\ell-1}<|z|\le2^{\ell}}|f(z)|\,dz\bigg|^2\frac{dydt}{t^{n+1}}\right)^{1/2}\notag\\
&\le C\bigg(\int_{2^{\ell-1}<|z|\le2^{\ell}}|f(z)|\,dz\bigg)
\left(\int_{\frac{|z|}{2^{j+2}}}^\infty 2^{jn}\frac{dt}{t^{2n+1}}\right)^{1/2}\notag\\
&\le C\cdot2^{{3jn}/2}\bigg(\int_{2^{\ell-1}<|z|\le2^{\ell}}\frac{|f(z)|}{|z|^n}dz\bigg).
\end{align}
Furthermore, by Minkowski's inequality, (3.3) and (4.10), we have
\begin{equation*}
\begin{split}
\big\|\mathcal S_{\beta,2^j}(f_3)\chi_k\big\|_{L^q_{w_2}}&\le\sum_{\ell=k+2}^{\infty}\big\|\mathcal S_{\beta,2^j}(f\chi_{\ell})\chi_k\big\|_{L^q_{w_2}}\\
&\le C\cdot2^{{3jn}/2}\sum_{\ell=k+2}^{\infty}
\frac{[w_2(B_k)]^{1/q}}{[w_2(B_{\ell})]^{1/q}}\big\|f\chi_{\ell}\big\|_{L^q_{w_2}}.
\end{split}
\end{equation*}
Therefore
\begin{equation*}
\begin{split}
J_3&\le C\sum_{k\in\mathbb Z}\big[w_1(B_k)\big]^{{\alpha p}/n}\bigg(\big\|\mathcal S_{\beta}(f_3)\chi_k\big\|_{L^q_{w_2}}+\sum_{j=1}^\infty 2^{-j\lambda n/2}\big\|\mathcal S_{\beta,2^j}(f_3)\chi_k\big\|_{L^q_{w_2}}\bigg)^p\\
&\le C\sum_{k\in\mathbb Z}\big[w_1(B_k)\big]^{{\alpha p}/n}\bigg(\sum_{\ell=k+2}^{\infty}
\frac{[w_2(B_k)]^{1/q}}{[w_2(B_{\ell})]^{1/q}}\big\|f\chi_{\ell}\big\|_{L^q_{w_2}}\bigg)^p
\times\bigg(1+\sum_{j=1}^\infty 2^{-j\lambda n/2}\cdot2^{{3jn}/2}\bigg)^p\\
&\le C\sum_{k\in\mathbb Z}\big[w_1(B_k)\big]^{{\alpha p}/n}\bigg(\sum_{\ell=k+2}^{\infty}
\frac{[w_2(B_k)]^{1/q}}{[w_2(B_{\ell})]^{1/q}}\big\|f\chi_{\ell}\big\|_{L^q_{w_2}}\bigg)^p,
\end{split}
\end{equation*}
where the last inequality also holds since $\lambda>3$. Following along the same lines as in Theorem 1.1, we can also show that
\begin{equation*}
J_2\le C\big\|f\big\|^p_{\dot K^{\alpha,p}_q(w_1,w_2)}
\end{equation*}
and
\begin{equation*}
J_3\le C\big\|f\big\|^p_{\dot K^{\alpha,p}_q(w_1,w_2)}.
\end{equation*}
Summing up the above estimates for $J_1$, $J_2$ and $J_3$, we complete the proof of Theorem 1.3.
\end{proof}

\begin{proof}[Proof of Theorem 1.4]
Let $f\in \dot K^{\alpha,p}_q(w_1,w_2)$. We set $f(x)=f_1(x)+f_2(x)+f_3(x)$ as in Theorem 1.2, then for any given $\sigma>0$, we can write
\begin{equation*}
\begin{split}
&\sigma^p\cdot\sum_{k\in\mathbb Z}\big[w_1(B_k)\big]^{{\alpha p}/n}w_2\Big(\Big\{x\in C_k:\big|\mathcal G^*_{\lambda,\beta}(f)(x)\big|>\sigma\Big\}\Big)^{p/q}\\
\le\,&\sum_{i=1}^3\sigma^p\cdot\sum_{k\in\mathbb Z}\big[w_1(B_k)\big]^{{\alpha p}/n}
w_2\Big(\Big\{x\in C_k:\big|\mathcal G^*_{\lambda,\beta}(f_i)(x)\big|>\sigma/3\Big\}\Big)^{p/q}\\
=\,&J'_1+J'_2+J'_3.
\end{split}
\end{equation*}
Since $\lambda>\max\{q_2,3\}\ge\max\{q_2,{2q_2}/q\}$ when $q_2\le q$. Applying Chebyshev's inequality, Lemma 2.1 and (4.8), we obtain
\begin{equation*}
\begin{split}
J'_1&\le\sigma^p\cdot\sum_{k\in\mathbb Z}\big[w_1(B_k)\big]^{{\alpha p}/n}\left(\frac{3^q}{\sigma^q}\big\|\mathcal G^*_{\lambda,\beta}(f_1)\big\|^q_{L^q_{w_2}}\right)^{p/q}\\
&\le C\sum_{k\in\mathbb Z}\big[w_1(B_k)\big]^{{\alpha p}/n}\big\|f_1\big\|_{L^q_{w_2}}^p\\
&\le C\sum_{k\in\mathbb Z}\big[w_1(B_k)\big]^{{\alpha p}/n}\big\|f\chi_k\big\|_{L^q_{w_2}}^p\\
&\le C\big\|f\big\|^p_{\dot K^{\alpha,p}_q(w_1,w_2)}.
\end{split}
\end{equation*}
For the term $J'_2$, when $x\in C_k$, then it follows from (4.7), (4.9), (3.3) and the fact $\lambda>3$ that
\begin{equation*}
\begin{split}
\big|\mathcal G^*_{\lambda,\beta}(f_2)(x)\big|&\le \sum_{\ell=-\infty}^{k-2}\big|\mathcal G^*_{\lambda,\beta}(f\chi_{\ell})(x)\big|\\
&\le C\sum_{\ell=-\infty}^{k-2}\bigg(\big|\mathcal S_{\beta}(f\chi_{\ell})(x)\big|+\sum_{j=1}^\infty2^{-j\lambda n/2}\big|\mathcal S_{\beta,2^j}(f\chi_{\ell})(x)\big|\bigg)\\
&\le C\bigg(\sum_{\ell=-\infty}^{k-2}\frac{1}{|x|^n}\int_{2^{\ell-1}<|z|\le2^{\ell}}|f(z)|\,dz\bigg)
\bigg(1+\sum_{j=1}^\infty2^{-j\lambda n/2}\cdot2^{{3jn}/2}\bigg)\\
&\le C\sum_{\ell=-\infty}^{k-2}\frac{1}{|x|^n}\bigg(\int_{2^{\ell-1}<|z|\le2^{\ell}}|f(z)|\,dz\bigg)\\
&\le C \sum_{\ell=-\infty}^{k-2}\frac{|B_{\ell}|}{|B_k|}
\big[w_2(B_{\ell})\big]^{-1/q}\big\|f\chi_{\ell}\big\|_{L^q_{w_2}}.
\end{split}
\end{equation*}
For the last term $J'_3$, when $x\in C_k$, by using (4.7), (4.10), (3.3) and the fact that $\lambda>3$, we get
\begin{equation*}
\begin{split}
\big|\mathcal G^*_{\lambda,\beta}(f_3)(x)\big|&\le \sum_{\ell=k+2}^{\infty}\big|\mathcal G^*_{\lambda,\beta}(f\chi_{\ell})(x)\big|\\
&\le C\sum_{\ell=k+2}^{\infty}\bigg(\big|\mathcal S_{\beta}(f\chi_{\ell})(x)\big|+\sum_{j=1}^\infty2^{-j\lambda n/2}\big|\mathcal S_{\beta,2^j}(f\chi_{\ell})(x)\big|\bigg)\\
&\le C\bigg(\sum_{\ell=k+2}^{\infty}\int_{2^{\ell-1}<|z|\le2^{\ell}}\frac{|f(z)|}{|z|^n}dz\bigg)
\bigg(1+\sum_{j=1}^\infty2^{-j\lambda n/2}\cdot2^{{3jn}/2}\bigg)\\
&\le C\sum_{\ell=k+2}^{\infty}\int_{2^{\ell-1}<|z|\le2^{\ell}}\frac{|f(z)|}{|z|^n}dz\\
&\le C\sum_{\ell=k+2}^{\infty}\big[w_2(B_{\ell})\big]^{-1/q}\big\|f\chi_{\ell}\big\|_{L^q_{w_2}}.
\end{split}
\end{equation*}
The rest of the proof is exactly the same as that of Theorem 1.2, and we finally obtain
\begin{equation*}
J'_2\le C\big\|f\big\|^p_{\dot K^{\alpha,p}_q(w_1,w_2)}
\end{equation*}
and
\begin{equation*}
J'_3\le C\big\|f\big\|^p_{\dot K^{\alpha,p}_q(w_1,w_2)}.
\end{equation*}
Combining the above estimates for $J'_1$, $J'_2$ and $J'_3$, and then taking the supremum over all $\sigma>0$, we conclude the proof of Theorem 1.4.
\end{proof}

\end{document}